\documentclass{jmmrc}
\theoremstyle{definition}
\newtheorem{theo}{Theorem}[section]
\newtheorem{pro}[theo]{Proposition}
\newtheorem{coro}[theo]{Corollary}
\newtheorem{lem}[theo]{Lemma}

\newtheorem{exam}[theo]{Example}
\newtheorem{rem}[theo]{Remark}
\newtheorem{defi}[theo]{Definition}

\newcommand{\ndim}{\mbox{{\it n}-{\rm dim}}}
\newcommand{\kdim}{\mbox{{\it k}-{\rm dim}}}

\newcommand{\gdim}{\mbox{{\it G}-{\rm dim}}}
\newcommand{\hdim}{\mbox{{\it h}-{\rm dim}}}

\newcommand{\be}{\begin{enumerate}}
\newcommand{\ee}{\end{enumerate}}

\begin{document}

\title[Dimension symmetry]{On multiplication $fs$-modules and dimension symmetry}
	
	\author{S. M. Javdannezhad}
	\address{Sayed Malek Javdannezhad
		\newline Orcid number: 0000-0001-7752-8047
		\newline Department of Science
		\newline Shahid Rajaee Teacher Training University: Tehran, Iran
		\newline Tehran, Iran}
	\email{sm.javdannezhad@gmail.com}
	
	\author{S. F. Mousavinasab}
	\address{Sayedeh Fatemeh Mousavinasab
	\newline Orcid number: 0000-0001-6027-7587 
		\newline Department of Mathematics
		\newline Shahid Chamran University of Ahvaz
		\newline Ahvaz, Iran}
	\email{mousavinasabfa@gmail.com}
	
	\author{N. Shirali*}
	\address{Nasrin Shirali	
\newline Orcid number: 0000-0002-9907-752
		\newline Department of Mathematics
		\newline Shahid Chamran University of Ahvaz
		\newline Ahvaz, Iran}
	\email{shirali\_n@scu.ac.ir\\nasshirali@gmail.com}

\thanks{\hspace*{-0.60cm} {\footnotesize\noindent {$^*$}Corresponding author, 
ORCID: 0000-0002-9907-752,\\  
e-mail: shirali\_n@scu.ac.ir\\nasshirali@gmail.com}\\
{\footnotesize DOI: 10.xxxxx/jmmrc.2021.yyyyy.zzzz 
\hspace{3cm}\copyright{ the Author(s)}\\
How to cite: S. M. Javdannezhad, S. F. Mousavinasab, N. Shirali, {\it On multiplication $fs$-modules and dimension symmetry},
 J. Mahani Math. Res. Cent. 2022; 11(2): 1-xx.
		}}

\fancyhead[CO]{\tiny{Dimension symmetry   {JMMRC Vol. 11, No. 2 (2022)  }}}
\fancyhead[CE]{ S. M. Javdannezhad, S. F.  Mousavinasab, N. Shirali}

\maketitle

\begin{center}
	{\tiny Article type: Research Article \\(Received xx August 202x, Revised: xx October 202x, Accepted: xx June 202x)}\\
	{\tiny (Communicated by ....)}
\end{center}

\begin{abstract}
In this paper, we first study $fs$-modules, i.e., modules with finitely many small submodules.   
 We show that  every $fs$-module with finite hollow dimension  is Noetherian.
  Also, we prove that an $R$-module $M$ with finite Goldie dimension, is an $fs$-module if and only if $M = M_1 \oplus M_2$,  where $M_1$ is  semisimple  and $M_2$ is an $fs$-module with $Soc(M_2) \ll M$. 
  Then, we investigate multiplication $fs$-modules over commutative rings  
and we prove that the lattices of $R$-submodules of $M$ and $S$-submodules of $M$ are coincide, where $S=End_R(M)$. Consequently, $M_R$ and $_SM$ have the same  Krull (Noetherian, Goldie and hollow) dimension. Further, we prove that  for any self-generator multiplication  module $M$,  the $fs$-module as a right $R$-module and as a left $S$-module are equivalent.\\
 \keywords{Small submodules, $fs$-modules, Multiplication modules, Dimension symmetry.} 
\subject{Primary  16P60, 16P20, 16P40.}
\end{abstract}

\section{Introduction}
In this paper, we focus on modules with finitely many small submodules (briefly, $fs$-modules).  It is  well known that  $M$ is a semisimple and Noetherian module if and only if it  is  Artinian  with $Rad(M)=0$ (i.e., $0$ is the only small submodule of $M$). Thus every Artinian module  $M$ with $Rad(M)=0$, is Noetherian. Motivated by this, it is natural to ask: Is any Artinian $fs$-module,     Noetherian? 
We first try to answer to this question and then investigate multiplication $fs$-modules  over commutative rings . For this,  we study  some basic properties of  modules  with finitely many small submodules.
For instance, we show that  $M$ is an $fs$-module if and only if $Rad(M)$ has only finitely many submodules. 
 Also, we show that $fs$-modules are  closed under submodules and small quotients (i.e., every factor $\frac{M}{N}$, where $N$ is small in $M$). 
 We prove that  every $fs$-module with finite hollow dimension  is Noetherian.
 Actually, we extend the latter fact to a larger class of modules  (note, every  Artinian module has finite hollow dimension).
 In particular,  we show that an $R$-module $M$ with finite Goldie dimension is an $fs$-module if and only if  $M = M_1 \oplus M_2$,  where $M_1$ is  semisimple and $M_2$ is  an $fs$-module with $Soc(M_2) \ll M$.  Moreover, we  give some examples, to show that for an arbitrary module $M$, the properties of being an $fs$-module and to have finite hollow dimension are independent.
 Then, we focus on multiplication $fs$-modules over commutative rings. 
 In particular, for any $N \subseteq M$, we prove that $N$ is an $R$-submodule of $M$ if and only if it is an $S$-submodule of $M$, where $S=End_R(M)$. This immediately implies that the lattices of $R$-submodules of $M$ and $S$-submodules of $M$ are coincide.
 Consequently, $M_R$ and $_SM$ have the same Krull (Noetherian, Goldie, hollow) dimension. Also, we show that  for any self-generator multiplication  module $M$,  the concept of an $fs$-module as a right $R$-module and as a left $S$-module are equivalent.\\ 
   Throughout this article, all rings are  associative with non-zero identity and all modules are unital right modules.
 Let $R$ be a ring and $M$ be an $R$-module. $R$ is called local, if  it has a unique maximal right (equivalently, left) ideal. $M$ is called local if it has exactly one maximal submodule that contains  all its proper submodules. The notation $N \subseteq_e M$ (resp., $N \ll M$) will denote $N$ is an essential (resp., a small) submodule of $M$, that is,  $N \cap A \ne 0$ (resp., $N+ A \ne M$), for all non-zero (resp., proper) submodules $A$ of $M$. $M$ has finite Goldie (resp., hollow) dimension if for any ascending (resp., descending) chain $N_{1}\subseteq N_{2}\subseteq \cdots$ (resp., $N_{1}\supseteq N_{2}\supseteq \cdots $) of submodules of $M$ there exists an integer $n \geq 1$,  such that $N_n \subseteq_e N_k$ (resp., $ \frac{N_n}{N_k} \ll \frac{M}{N_k}$ )  for all $k \geq n$. $Soc(M)$ (resp., $Rad(M)$) will denote the socle (resp., radical) of $M$, i.e., the sum (resp., intersection) of all minimal (resp., maximal) submodules of $M$. Also, $Soc(M)$ (resp., $Rad(M)$) is equal to the intersection (resp., sum) of all essential (resp., small) submodules of $M$. If $M$ fails to have minimal (resp., maximal) submodules, we set $Soc(M)=0$ (resp., $Rad(M) = M$) and in case,  $M$ fails to have proper essential (resp., nonzero small) submodule, then $Soc(M)=M$ (resp., $Rad(M) = 0$). Hence, any module with non-trivial socle (resp., radical) has both minimal and proper essential (resp., maximal and non-zero small) submodules. Also, $J(R)$ denote the Jacobson radical of $R$, i.e., the intesection of all maximal right  ideals of $R$. Note that $J(R) = Rad(R_R) = Rad(_RR)$. Finally, $\kdim\,M$, $\ndim\,M$, $\gdim\,M$ and $\hdim\,M$ respectively, denote the  Krull dimension, the Noetherian dimension, the Goldie dimension and the hollow dimension of $M$. The Noetherian dimension is also known as the dual Krull dimension and $N$-dimension. For more details on these dimensions and undefined terms and notations, we refer to  \cite{al-sm, ch, goo, kh, ksh, le2, wis}.

\section{$fs$-modules and $fs$-rings}
  In this section, we give our definition of $fs$-modules and study their properties.

  \begin{defi}
 An $R$-module $M$ with only finitely many non-zero small submodules is said to be an $fs$-module. A ring $R$ is called a right (left) $fs$-ring if as a right (left) $R$-module, it is an $fs$-module.
  \end{defi}

  \begin{rem}
  Let $M$ be an $R$-module.
  \begin{enumerate}
    \item If $Rad(M) = 0$, then $M$ is an $fs$-module.
    \item If $Rad(M) = M$, then $M$ is not an $fs$-module. For this, note that every finite sum of small submodules is a small submodule, so if $M$ is an $fs$-module $M$, then $Rad(M)$ is a small submodule of $M$, hence $Rad(M) \ne M$. 
    \item If $M$ is an $fs$-module, then $M$ has at least one maximal submodule.
  \end{enumerate}
  \end{rem}

    The following easy access results show that the class of $fs$-modules are closed under submodules and  small quotients small quotients (i.e., every factor $\frac{M}{N}$, where $N$ is small in $M$), see also \cite[Propositions 2.5, 2.6]{ja-sh-4}. 

   \begin{pro}\label{3-10}
   The following are equivalent for any $R$-module $M$.
   \begin{enumerate}
   \item $M$ is an  $fs$-module.
   \item  Every submodule of $M$ is an $fs$-module.
   \item  Every small quotient of $M$  is an  $fs$-module.
  \end{enumerate}
  \end{pro}

  \begin{pro}\label{1-1}
  Let $M$  be an $R$-module. Then $M$ is an $fs$-module if and only if $Rad(M)$ has only finitely many submodules.
  \end{pro}

  \begin{proof}
    If $Rad(M)$ has only finitely many submodules, then $M$ necessarily is an $fs$-module, since $Rad(M)$ contains every small submodule of $M$. Conversely, let $M$ be an $fs$-module. Since $Rad(M)$ equals to the sum of all small submodules of $M$, it follows that $Rad(M)$ is small in $M$, so is every submodule that is contained in  $Rad(M)$. Hence,  $Rad(M)$ has only finitely many submodules.
 \end{proof}

 The following result is  also in \cite[Corollary 2.7]{ja-sh-4}.

 \begin{coro}\label{*'}
 If $M$ is an $fs$-module, then
    \begin{enumerate}
 \item $Rad(M)$ has finite length, so it is both Artinian and Noetherian.
 \item $M$ is Noetherian (Artinian) if and only if  $\frac{M}{Rad(M)}$ is Noetherian (Artinian).
    \end{enumerate}
\end{coro}

\begin{proof}
 (1) By the previous proposition, $Rad(M)$ has only finitely many submodules,  thus  it has finite length and so it is both Artinian and Noetherian.\\
 (2) By part (1), it is evident.
   \end{proof}
    \begin{coro}\label{c1}
  Every local $fs$-module $M$ is both Noetherian and Artinian. In particular, every local and right $fs$-ring $R$ is both right Noetherian and right Artinian.
  \end{coro}

 \begin{proof}
  Let $K$ be the unique maximal submodule of $M$, then $Rad(M)=K$ and so $\frac{M}{Rad(M)}= \frac{M}{K}$ is  simple. Hence, $\frac{M}{Rad(M)}$ is both   Noetherian and Artinian and by Corollary \ref{*'}, we  are done. 
 \end{proof}
 
  \begin{rem}
    If $S$ is a non-zero small submodule of $M$, then so is every non-zero submodule of $M$, contained in $S$. Hence, every minimal member with respect to inclusion, in the set of non-zero small submodules of $M$, necessarily is a minimal submodule of $M$. In other words, every minimal non-zero small submodule is a minimal submodule. From this, it follows that a  $\lq\lq$minimal non-zero small submodule" or equivalently a $\lq\lq$small minimal submodule", is precisely a non-zero submodule which is $\lq\lq$both  minimal and small submodule" of $M$.
   \end{rem}

  \begin{pro}\label{3-100}
  Let $M$ be an $fs$-module with $Rad(M) \ne 0$. Then $M$ has a minimal and small submodule, that is, 
  $Soc(Rad(M)) \ne 0$.
  \end{pro}

  \begin{proof}
  The set of all non-zero small submodules of $M$ is finite, so it  has a minimal element, say $S$,  which is also a minimal submodule of $M$, by the above remark. Hence $0 \ne S \subseteq Rad(M) \cap Soc(M) = Soc(Rad(M))$, by  \cite[21.2(2)]{wis}.
 \end{proof}

  \begin{rem}\label{R3}
    For every $R$-module $M$, it follows from  \cite[Lemma 2.2(9)]{al-sh}, that $Soc(Rad(M) )$ is a small submodule of $M$. Hence, a semisimple submodule $S$ of $M$ is small in  $M$ if and only if  $ S \subseteq Rad(M)$.
  \end{rem}

  \begin{lem} \label{r5}
    Every minimal submodule of an $R$-module $M$ is either small or a direct summand. For this, let $S$ be a minimal submodule of $M$. If $S \subseteq Rad(M)$, then $S \subseteq Soc(Rad(M))$, so it is small by previous remark. If $S \nsubseteq Rad(M)$, then $S \nsubseteq K$  for some maximal submodule $K$ of $M$. Hence, $S \cap K =0$ and $S \oplus K = M$.
\end{lem}

 Now, we give our structure theorem for $fs$-modules with finite Goldie dimension, see also  \cite[Theorem 2.20]{ja-sh-4}. 

\begin{theo}\label{3-18}
    Let $M$ be an $R$-module with finite Goldie dimension. Then $M$ is an $fs$-module if and only if $M = M_1 \oplus M_2$, where $M_1$ is semisimple  and $M_2$ is an $fs$-module such that $Soc(M_2) \ll M$.
\end{theo}

\begin{proof}
    First of all, note that since $M$ has finite Goldie dimension, it follows that $M$ has finitely many minimal submodules. Let $M$ be an $fs$-module. If every minimal submodule of $M$ is small in $M$, then $Soc(M)$ is a finite direct sum of small submodules. It follows that $Soc(M) \ll M$, so in this case $M=0 \oplus M$ and we are done. In other case,  $M$ has some minimal submodules which are non-small,  let $N_1$  be a minimal and non-small submodules of $M$.  By Lemma \ref{r5} there exists a submodule $K_1$ of $M$ such that  $N_1 \oplus K_1 = M$. If $K_1$ has no minimal and non-small submodule, then $Soc(K_2) \ll M$. Set $M_1 = N_1$ and $M_2 = K_2$ and we are done. Otherwise $K_1 = N_2 \oplus K_2$, for some minimal and non-small submoule $N_2$ of $K_1$ and some submodule $K_2$ of $K_1$ and so $M = N_1 \oplus N_2 \oplus K_2$. 
    Since $Soc(M)$ has  finite Goldie dimension, so $M$ has finite number of minimal submodules, finally we have $M = N_1 \oplus N_2 \oplus \cdots \oplus N_m \oplus K_m $ such that $N_i$'s are minimal and non-small and all minimal submodules of $K_m$ are small, that is, $Soc(K_m) \ll M$.
 Let $M_1 = N_1 \oplus N_2 \oplus \cdots \oplus N_m $  and $M_2 = K_m$, then we are done. To prove the converse,  since $M_1$ is  semisimple, it follows from \cite[Lemma 2.2(6)]{al-sh} that every small submodule of  $M$ is in the  form of $0\oplus S_2$ such that    $ S_2\ll M_2$, and so    $M$ is an $fs$-module. 
 \end{proof}

\begin{defi}
  An $R$-module $M$ is called homogeneous, if every non-zero submodule of  $M$ has a non-zero small submodule.
\end{defi}

  The next result is devoted to homogeneous $fs$-modules.

 \begin{pro}
 Let $M$ be a homogeneous $fs$-module. Then $Soc(M)$ is essential in $M$.
 \end{pro}

 \begin{proof}
 Let $N$ be a non-zero submodule of $M$. By Proposition \ref{3-10}, $N$ is an $fs$-module. So it follows from Proposition \ref{3-100} that  $0\neq Soc(N)\subseteq N\cap Soc(M)$, and hence we are done.
 \end{proof}
 We recall that an $R$-module $M$ is called finitely embedded if $Soc(M)$ is a finitely generated and essential submodule of $M$.

 \begin{coro}
   Let $M$ be a homogeneous $fs$-module. If $Soc(M)$ is finitely generated, then $M$ is finitely embedded.
 \end{coro}

 \begin{proof}
    This follows from  previous proposition and definition of finitely embedded modules.
 \end{proof}

 Recall that $R$ is called a duo ring, if every one-sided ideal of $R$ is an ideal of $R$. Also, a commutative ring  $R$ is locally Noetherian if  the localization $R_M$ is Noetherian for every maximal ideal $M$ of $R$.

  \begin{pro}
  Let $R$ be a locally Noetherian or a Noetherian duo ring and  $M$ be a homogeneous $fs$-module  for which $Soc(M)$ is finitely generated. Then $M$ is Artinian.
  \end{pro}

  \begin{proof}
  If $R$ is locally Noetherian or a Noetherian duo ring, then every finitely embedded module is Artinian, see \cite [Theorem 2]{va1} and  \cite [Theorem 2.4]{kh}.
  \end{proof}

   We cite the following important fact from  \cite[Corollary 1.10]{sa-va}.

 \begin{theo}\label{3-1}
An  $R$-module $M$ with   $Rad(M)=0$ has finite hollow dimension if and only if it is  finitely generated  semisimple.
 \end{theo}
 The following result is also in \cite[Theorem 2.13]{ja-sh-4}. 
 
  \begin{theo}\label{3-6}
    Let $M$ be an $fs$-module with finite hollow dimension over a ring $R$. The following holds.
    \begin{enumerate}
        \item $M$ is Artinian.
        \item $M$ is  Noetherian.
        \item $Rad(M)$ is finitely generated.
        \item $M$ has finite Goldie dimension.
        \item $M$ has a finite composition series.
        \item $M$ has finite length.
   \end{enumerate}
  \end{theo}

  \begin{proof}
  It follows from  previous theorem that $\frac{M}{Rad(M)}$ is both Artinian and Noetherian. Then so is $M$,  by the part $(2)$ of
   Corollary \ref{*'}. It follows that $Rad(M)$ is finitely generated and $M$ has finite Goldie dimension. Moreover, such a module $M$ has a finite composition series, so it has finite length. 
   
   \end{proof}

   \begin{coro}\label{3-666}
    Let $M$ be an $fs$-module. Then $M$ is Artinian if and only if it has finite hollow dimension.
   \end{coro}

   \begin{proof}
   Every Artinian module clearly has finite hollow dimension. The converse is true by the previous theorem.
   \end{proof}

   \begin{pro}
   Let $R$ be a semiprime ring. The following statements are equivalent.
   \begin{enumerate}
    \item $R$ is a right $fs$-ring with finite right hollow dimension.
    \item $R$ is a left $fs$-ring with finite left hollow dimension.
    \item $R$ is a semisimple ring.\\
    Moreover, if $R$ is local,  then these are equivalent to  \item $R$ is a division ring.
\end{enumerate}
  \end{pro}

  \begin{proof}
    By  previous corollary, $(1)$ is equivalent to $R$ is right Artinian and $(2)$ is equivalent to $R$ is left Artinian. Since $R$ is semiprime, these are equivalent to $R$ is semisimple, see \cite[Corollary 3.17]{goo}. In case $R$ is local, then it is a semisimple local ring. Equivalently, $J(R) = 0$  is the unique maximal right (left) ideal of $R$, that is, $R$ is a division ring.
  \end{proof}

 \begin{pro}
 Let $R$ be a ring. Then at the same time, $R[x]$ cannot have finite right (left) hollow dimension and to be a right (left) $fs$-ring.
 \end{pro}

 \begin{proof}
  Since $(x) \supsetneq (x^2) \supsetneq (x^3) \supsetneq \cdots $ is an infinite descending chain of ideals in $R[x]$, it follows that $R[x]$ is not right (nor left) Artinian.  By Corollary \ref{3-666}, we are done.
 \end{proof}

 \begin{defi}
    An $R$-module $M$  is said to be an $AB5^*$ module, if for every submodule $B$ and inverse system $\{A_{i}\}_{i\in I}$ of submodules of $M$,
    $$B +\bigcap_{i\in I}A_{i}=\bigcap_{i\in I}(B+A_{i}).$$
 \end{defi}
 Artinian modules and linearly compact modules are  $AB5^{*}$, see \cite[29.8]{wis}. We also recall that an $R$-module $M$ is called $q.f.d$ (i.e., quotient finite  dimentional) if every factor module of $M$ has finite Goldie dimension.

  \begin{lem}
    Let $M$ be an $AB5^*$  $fs$-module. If $M$ has Krull dimension, then $M$ is both Noetherian and Artinian.
  \end{lem}
  \begin{proof}
  By \cite[Proposition 1.3]{le2},  $M$ is  $q.f.d$ if and only if  every submodule of $M$ has finite hollow dimension. Hence, if $M$ is both an $fs$-module and a $q.f.d$-module, then by Theorem \ref{3-6}, it is both Noetherian and Artinian.
  \end{proof}

  \begin{exam}\label{0123}
    For an arbitrary $R$-module $M$, to be an $fs$-module and to have finite hollow dimension are independent. For this, we give some examples.
    \begin{enumerate}
        \item  For an $fs$-module with finite hollow dimension, we refer to every finite $R$-module $M$ with $Rad(M)=0$.

        \item For an $fs$-module with infinite hollow dimension, we refer to $\mathbb Z$, the ring of integers as  $\mathbb Z$-module.
        \item For a  non-$fs$-module with finite hollow dimension, we refer to $\mathbb Z_{p^{\infty}}$ as  $\mathbb Z$-module, where $p$ is a prime number.

        \item For a non-$fs$-module with infinite hollow dimension, we refer to  $\mathbb Q$, the set of rational numbers  as  $\mathbb Z$-module. 
    \end{enumerate}
  \end{exam}
 \begin{rem} 
  By  the part $4$ of above example we infer that a proper  essential extension of an $fs$-module need not to be an $fs$-module.
 \end{rem}
 
  \begin{exam}
    For each $n \in \mathbb N$, there exists a non-$fs$-module $M$ such that  $\ndim\,M = n$.
    For this, let $F$ be a field and $R=F[x_1,x_2,...,x_n]$. Then  $R$ is a commutative Noetherian ring and every maximal ideal $M$ of $R$, is
    exactly of  rank $n$, that is, there exists a chain of  length $n$  of prime ideals descending from $M$, but no longer chain. Now, let
    $S$ be a simple $R$-module and  $A = E(S)$ be the injective envelope of $S$. By \cite[Theorem 2]{va1}, $A$ is Artinian and by \cite[Poroposion 5]{ch}, $\ndim\, A = Rank(M) = n > 1$, where $M$ is a maximal ideal of $R$ such that $S \cong \frac{R}{M}$. This implies that $A$ is not an $fs$-module, since by Corollary \ref{3-666}, for every  $fs$-module with finite hollow dimension, must be  Noetherian. This example, also shows that there exist Artinian modules with Noetherian dimension of any natural number.
   \end{exam}

  \begin{defi}
    An $R$-module $M$ is called a $us$-module, if it has a unique non-zero small submodule. $R$ is called a right $us$-ring, if as an $R$-module it is a $us$-module.
 \end{defi}

 The following well-known fact is due to Brauer.

\begin{theo}\label{Br}
    If $A$ is a minimal right ideal of a ring $R$, then either $A^2=0$ or $A=eR$  for some idempotent $e \in R$.
\end{theo}

\begin{theo}\label{0}
The following statements are equivalent for any ring $R$.

\begin{enumerate}
\item $R$ is a right $us$-ring.
\item $J=J(R)$ is minimal as a  right ideal of $R$ and $J^2 =0$.
\item Each right ideal $A$ of $R$ is either minimal or non-small.
\end{enumerate}
\end{theo}

\begin{proof}
 $(1) \Rightarrow (2)$. Clearly $J$ is the unique small right ideal of $R$, hence it is minimal as a right ideal (for, if $S$ is a non-zero right ideal of $R$ contained in $J$, then  $S$ is small by \cite[Lemma 2.2(1)]{al-sh},  so $S=J$ ). On the other hand, $J$  have no any idempotent element, so  $J^2 = 0$ by Theorem \ref{Br}.\\
 $(2) \Rightarrow (3)$. Let $A$ be a non-zero right ideal of $R$. If $A$ is not minimal, then $A \ne J$, hence it is non-small. \\
  $(3) \Rightarrow (1)$. Since $J(R)$ is small as a right (and left) ideal of $R$, it is minimal right ideal of $R$  by $(3)$. Hence,  $R$ is an $us$-ring.
\end{proof}

\begin{defi}
    An $R$-module $M$ is said to be dual-local if it has a unique minimal submodule, that is, $M$ has a non-zero submodule $N$ that contained in every non-zero submodule of $M$.
\end{defi}

  In \cite{kmsh}, a ring with a unique essential proper right ideal is called a right $ue$-ring. Similarly, an  $R$-module with a unique
  proper essential submodule is called a $ue$-module. The following theorem seems to be interesting.

\begin{theo}\label{dloc}
     The following statements are equivalent for an $R$-module $M$.
    \begin{enumerate}
        \item $M$ is a local $us$-module.
        \item $M$ is a dual-local $ue$-module.
        \item $M$ has a unique non-trivial submodule.
    \end{enumerate}
    Moreover, for such a module $Soc(M)=Rad(M)$.
\end{theo}

\begin{proof}
    $(3)\Rightarrow (1)$ and $(3)\Rightarrow (2)$ are clear.\\
    $(1)\Rightarrow (3)$
    Since $M$ is local, $Rad(M)$ is the unique maximal submodule of $M$. Since $M$ is a $us$-module, $Rad(M)$ and it the unique non-zero small submodule of $M$. If $0 \ne A$ is a submodule of $Rad(M)$, then $A$ is small and so $A=Rad(M)$. Hence $Rad(M)$ is also a minimal submodule of $M$ and so $Rad(M)$ is both a maximal and a minimal submodule of $M$. Again, let $A$ be a non-trivial submodule of $M$. Then $A \subseteq Rad(M)$ and by minimality of $Rad(M)$ we have $A = Rad(M)$. This shows that $Rad(M)$ is the unique non-trivial submodule of $M$ and we are done.\\ 
  $(2)\Rightarrow(3)$ Since $M$ is dual-local, $Soc(M)$ is the unique minimal submodule of $M$. Since $M$ is a $ue$-module, $Soc(M)$ is also the unique essential submodule of $M$. If $ A$ is a proper submodule of $M$ such that $Soc(M) \subseteq A$, then $A$ is essential and so $A=Soc(M)$. Hence, $Soc(M)$ is also a maximal submodule of $M$ and so $Soc(M)$ is both a minimal and a maximal submodule of $M$. Again, let $A$ be a non-trivial submodule of $M$. Then $Soc(M) \subseteq A$ and by maximality of $Soc(M)$, we have $A = Soc(M)$. This shows that $Soc(M)$ is the unique non-trivial submodule of $M$ and we are done.\\
    Moreover $(3)$ implies that $Soc(M)=Rad(M)$ is just the unique non-trivial submodule of $M$.
\end{proof}

\begin{coro}
    Let $R$ be a local and right (left) $us$-ring. Then as a right (left) ideal, $J(R)$ is both minimal and maximal, hence it is the unique non-trivial right ($2$-sided) ideal of $R$.
\end{coro}

\begin{defi}
    An $R$-module $M$ with only finitely many small and minimal (resp., unique small and minimal) submodules is called an  $fsm$-module (resp., $usm$-module). A ring $R$ with only finitely many small and minimal right ideals is called
    a right $fsm$-ring (resp., $usm$-ring).
\end{defi}

\begin{rem}
    Note that every $fs$-module is an $fsm$-module, but  the converse is not true ingeneral.  For example, $\mathbb Z_{p^{\infty}}$  as a $\mathbb Z$-module is an $fsm$-module which is not an $fs$-module. More generally, for every Artinian module $M$ which is not semisimple, we have  $Rad(M) \neq 0$, so $M$ has at least one small and minimal submodule. Since $Soc(M)$ is  Artinian, so it is finitely generated, hence $M$ is an $fsm$-module.
\end{rem}

\begin{pro}
    Let $M$ be an $R$-module. Then  $M$  is  an $fsm$-module if and only if  $Rad(M)$ has only finitely many minimal submodules.
\end{pro}

\begin{proof}
    In case that $Rad(M) = 0$, we need no explanation. For the other case, note that from Remark \ref {R3}, it follows that $M$ has only finitely many small and minimal submodules if and only if $Rad(M)$ has only finitely many minimal submodules.
\end{proof}

\section{Multiplication modules and dimension symmetry}
Troughout this section, $R$ is a commutative ring.
The concept of multiplication modules has been studied in  many articles, see for example \cite {ch-sm, nao}. An $R$-module $M$ is called multiplication if for every submodule $N$ of $M$, there exists an ideal $I$ of $R$ such that $N = MI$. In this case, we can take $I = (N:M) = ann(\frac{M}{N}) = \{r \in R : Mr \subseteq N\}$. The class of multiplication modules contains all projective ideals, all cyclic modules, all finitely generated distributive modules and all ideals $eR$, where $e$ is an idempotent. In this section, we focus on multiplication $fs$-modules.
 We recall that  if $M$ is an $R$-module and  $S=End_R(M)$, then $_SM_R$, that is, $M$ is an $(S-R)$-bimodule.
Also,  an $R$-submodule $X$ of $M$ is called fully invariant provided it is also an $S$-submodule of $M$, or equivalently, $f(X) \subseteq X$, for every $f \in S=End_R(M)$. \\

    We cite the following facts from \cite{nao} and \cite[Lemma 1]{ch-sm}.

  \begin{pro}\label{p4.1}
    Let $M$ be a multiplication $R$-module. Then $S=End_R(M)$ is a commutative ring.
  \end{pro}

  \begin{pro}\label{c4.1}
    Let $M$ be a multiplication module. Then every submodule of $M$ is fully invariant.
\end{pro}

 Previous proposition says that every $R$-submodule of a multiplication module is an $S$-submodule, where $S=End_R(M)$. In the next theorem we show that not only for multiplication $R$-modules, but also for every $R$-module $M$, every $S$-submodule is an $R$-submodule.

  \begin{pro}\label{srsub}
    Let $M$ be an $R$-module and $S=End_R(M)$. Then every $S$-submodule of $M$ is an $R$-submodule.
\end{pro}
\begin{proof}
    Let $N$ be an $S$-submodule of $M$. It suffices to show that  $Nr \subseteq N$, for any $r \in R$. For this, let $r \in R$ and define $f: M_R\rightarrow M_R$ by $f(m) = mr$, for each $m \in M$. It is easy to see that $f \in S=End_R(M)$. Hence, $Nr = f(N) \subseteq N$, that is,  $N$ is an $R$-submodule of $M$.
\end{proof}

The next theorem can be one of the major theorems in the theory of commutative rings.

\begin{theo}\label{maj}
   Let $M$ be a multiplication $R$-module, $N \subseteq M$ and $S=End_R(M)$. Then
   \begin{enumerate}
    \item $N$ is an $R$-submodule of $M$ if and only if it is an $S$-submodule of $M$.
    \item The lattices of $R$-submodules of $M$ and $S$-submodules of $M$ are coincide.
   \end{enumerate}
\end{theo}

\begin{proof}
    $(1)$ It comes out from Propositions \ref{c4.1} and \ref{srsub}.\\
    $(2)$ It follows by the part $(1)$.
\end{proof}

\begin{coro}
 Let $M$ be a multiplication $R$-module, $N \subseteq M$ and $S=End_R(M)$. Then $N$ is an essential (resp., a small) $R$-submodule of $M$ if and only if it is an essential (resp., a small) $S$-submodule of $M$.
\end{coro}

\begin{proof}
 It follows from   of Theorem \ref{maj}$(1)$.
\end{proof}

In \cite{ja-sh-3}, as  an special case of the Krull symmetry
property, see also the appendix of \cite{goo},  we study
$R$-modules $M$ for which $\kdim\,M_R = \kdim\,_SM$, where
$S=End_R(M)$. It is proved that if $R$ is an  $FBN$ ring and $M$
a fully bounded $NPG$ module (i.e.,  Noetherian, projective and
generator), then  $\kdim\,M_R = \kdim\,_SM$.  The next theorem
shows that for multiplication modules, this symmetry also holds,
for the Noetherian, Goldie and hollow dimensions and it also
holds for the properties of being  $\alpha$-DICC, $\alpha$-short
and $\alpha$-Krull. For more details on these latter concepts, we
refer the reader, respectively to \cite{ka3}, \cite {d-k-sh} and
\cite{ja-sh-2}.

\begin{theo}
Let  $M$ be a multiplication $R$-module and $S=End_R(M)$. Then
\begin{enumerate}
    \item $\gdim\,M_R = \gdim\,_SM$.
    \item $\hdim\,M_R = \hdim_SM$.
    \item $\kdim\,M_R = \kdim\,_SM$.
    \item $\ndim\,M_R = \ndim\,_SM$.\\
    Moreover, for every ordinal $\alpha$,
    \item $M_R$ is $\alpha$-DICC if  and only if $_SM$ is $\alpha$-DICC.
    \item $M_R$ is $\alpha$-short if and only if $_SM$ is $\alpha$-short.
    \item $M_R$ is $\alpha$-Krull if and only if $_SM$ is $\alpha$-Krull.
\end{enumerate}
\end{theo}
\begin{proof}
All these naturally come out from  of Theorem \ref{maj}$(2)$.
\end{proof}

 We recall that an $R$-module  $M$ is called self-generator if for each submodule $X$ of $M$, there exists  $\Delta \subseteq S=End_R(M)$,  such that
  $N= \sum_{f \in \Delta} f(M)$.
  For any $X\subseteq M$, we set $I_X =\{ f\in S: f(M) \subseteq X\}$.

 \begin{pro}\label{l4.1}
    Let $M$ be a self-generator multiplication $R$-module  and $S= End_R(M)$. Then  $M$ is a multiplication $S$-module.
  \end{pro}
  \begin{proof}
    Let $X \subseteq M$ be an $S$-submodule of $M$. Then $X$ is also an $R$-submodule of $M$. Also, $S$ is a commutative ring, by Proposition \ref{p4.1}.  Now, we may invoke \cite[Lemma 3.4(1)]{j-sh-1} to see that $I_X M = X$ and this completes the proof.
  \end{proof}

 We conclude this paper with the next result that seems to be interesting.

\begin{theo}
Let $M$ be a self-generator multiplication $R$-module and $S=End_R(M)$. Then $M_R$ is an $fs$-module if and only if $_SM$ is an $fs$-module
\end{theo}

\begin{proof}
    It follows from Theorem \ref{maj}$(2)$.\\
    
    \end{proof}
     
     \section{Author Contributions} 
Conceptualization, N. Shirali.; methodology, N. Shirali, S.F. Mousavinsab, S.M. Javdannezhad; 
investigation, N. Shirali, S.F. Mousavinsab, S.F. Javdannezhad;
 writing—review and editing, N. Shirali;
 funding N. Shirali;  acquisition, N. Shirali, S.F. Mousavinsab, S.M. Javdannezhad; All authors have read and agreed to the published version of the manuscript. 
\section{Acknowledgments}
The authors would like to thank the referees for reading the article carefully and giving
useful comments and efforts towards improving the manuscript.  
  
  \section{Funding}
The third author is grateful to the Research Council of Shahid Chamran University of Ahvaz, Iran,  for financial support , and Grant No. SCU.MM1400.473.


\begin{thebibliography} {50}
 \bibitem{al-sm}
T. Albu,   P. F. Smith, {\it Dual Krull dimension and  duality},    Rocky Mountain J. Math.,  29 (1999), 1153--1164.

  \bibitem{al-sh}
   A. R. Alehafttan, N. Shirali, {\it On the small Krull dimension},    Comm. Algebra, 46(5) (2018), 2023--2032.

  \bibitem{ch}
  L. Chambless, {\it Dimension and N-critical modules, application to Artinian modules},    Comm. Algebra,   8(16) (1980), 1561--1592.

 \bibitem{ch-sm}
 C. W. Choi,  P.  F.  Smith,  {\it On endomorphisms of multiplication modules},    J. Korean Math. Soc.,   31(1) (1994), 89--95.

 \bibitem{d-k-sh}
 M. Davoudian, O. A. S. Karamzadeh, N. Shirali, {\it On $\alpha$-short modules},   Math. Scan., 114 (2016), 26--37.

\bibitem{goo}
 K. R. Goodearl,  R. B. Warfield, {\it An Introduction to Noncommutative Noetherian Rings}, Cambridge University Press,
Cambridge, UK,  (1989).

\bibitem{kh}
 J. Hashemi,  O. A. S. Karamzadeh, N. Shirali, {\it Rings over which the Krull dimension and the Noetherian dimension of all  modules coincide}, Comm. Algebra,   37 (2009), 650--662.

 \bibitem{j-sh-1}
S. M. Javdannezhad, N.  Shirali,  {\it  The Krull dimension of certain semiprime  modules versus their $\alpha$-shortness},  Mediterr. J. Math., 15 (2018), 116.

\bibitem{ja-sh-2}
S. M. Javdannezhad, N.  Shirali, {\it On the class-ication of $\alpha$-Krull modules},  JP. J. Math., 40(1) (2018), 1--12.

\bibitem{ja-sh-3}
 S. M. Javdannezhad,  N.  Shirali,  {\it On fully bounded modules and Krull symmetry}.  East-West J. of Mathematics, 22(2) (2020), 174--181.

\bibitem{ja-sh-4}
S. M. Javdannezhad, N.  Shirali, {\em  On modules with only finitely many small submodules},
52nd Annual Iranian Mathematics Conference, Shahid Bahonar University of Kerman, Kerman, Iran, 30 August- 02 September, 2021.

\bibitem{ksh}
O. A. S.  Karamzadeh,   N. Shirali, {\it On the countability of Noetherian dimension of modules}, Comm. Algebra, 32(10)  (2004), 4073--4083.

 \bibitem{kmsh}
O. A. S.  Karamzadeh,  M. Motamedi,  S. M. Shahrtash, {\it On rings with a unique proper essential right ideal},  Fundamenta Mathematicae, 183 (2004),  229--244.

 \bibitem{ka3}
  O. A. S.  Karamzadeh, M. Motamedi,   {\it On $\alpha$-DICC modules},  Comm. Algebra, 22 (1994), 1933--1944.

 \bibitem{le2}
  B.  Lemonnier, {\it  Dimension de Krull et codeviation, application au theorem
  d\'{E}akin}. Comm. Algebra, 6 (1978):1647--1665.
  
\bibitem{Lomp1}
 C. Lomp, {\it  On dual Goldie dimension}, Diplomarbeit (M.Sc. Thesis), HHU Doesseldorf, Germany, (1996).

   \bibitem{nao}
 A. G. Naoum,  {\it A note on projective modules and multiplication modules},  Beitrage
 zur Algebra und Geometry,   32 (1991), 27--32.

\bibitem{sm}
 P.  F. Smith,   {\it Multiplication modules},  Comm. Algebra, 16(4) (1988), 755--779

\bibitem{va1}
  P. Vamos,   {\it   The dual of the notion of finitely generated},  J. London Math. Soc., 43 (1968), 643--646.

  \bibitem{sa-va}
 B. Sarath, K. Varadarajan,  {\it Dual Goldie dimension II},  Comm. Algebra,  7 (1979), 1885--1899.

\bibitem{wis}
 R. Wisbauer, {\em Foundations of module and ring theory}, Gordon and Breach, Reading, Philadelphia, (1991).
 \end{thebibliography}
\end{document}